\documentclass[12pt,english]{amsart}

\input xy
\xyoption{all}

\usepackage[latin1]{inputenc}
\usepackage{amsfonts,amsmath,amssymb,amsthm}

\newcommand{\Cc}{\mathbb{C}} 
\newcommand{\Pp}{\mathbb{P}}

\newcommand{\Nn}{\mathbb{N}}
\newcommand{\Zz}{\mathbb{Z}}
\newcommand{\Qq}{\mathbb{Q}} 
\newcommand{\Ff}{\mathbb{F}}

{\theoremstyle{plain}
\newtheorem{theorem}{Theorem}[section]    

\newtheorem{twisting lemma}[theorem]{Twisting lemma}

\newtheorem{lemma}[theorem]{Lemma}       
\newtheorem{corollary}[theorem]{Corollary}   

}
{\theoremstyle{remark}

}

\font\sevenrm=cmr10 scaled 700

\def\sep{{\scriptsize\hbox{\rm sep}}}
\def\spe{{0}}

\def\ur{{\scriptsize\hbox{\rm ur}}}

\def\tr{{\hbox{\sevenrm tr}}}

\def\Gabs{\hbox{\rm G}}

\def\Gal{\hbox{\rm Gal}}

\def\cm{\hbox{\hbox{\rm C}\kern-5pt{\raise 1pt\hbox{$|$}}}}

\def\lhfl#1#2{\smash{\mathop{\hbox to 12mm{\leftarrowfill}}
\limits^{#1}_{#2}}}

\def\rhfl#1#2{\smash{\mathop{\hbox to 12mm{\rightarrowfill}}
\limits^{#1}_{#2}}}

\def\build#1_#2^#3{\mathrel{
\mathop{\kern 0pt#1}\limits_{#2}^{#3}}}

\def\htrait#1#2{\smash{\mathop{\hbox to 12mm{\hrulefill}}
\limits^{#1}_{#2}}}

\def\sxbullet{{\raise 2pt\hbox{\bf .}}}

\begin{document}

\title[Specialization results in Galois theory]{Specialization results in Galois theory}

\author{Pierre D\`ebes}

\author{Fran\c cois Legrand}

\email{Pierre.Debes@math.univ-lille1.fr}

\email{Francois.Legrand@math.univ-lille1.fr}

\address{Laboratoire Paul Painlev\'e, Math\'ematiques, Universit\'e Lille 1, 59655 Villeneuve d'Ascq Cedex, France}



\date{\today}

\begin{abstract} 
The paper has three main applications. The first one is this Hilbert-Grunwald statement.
If $f:X\rightarrow \Pp^1$ is a degree $n$ $\Qq$-cover with monodromy group $S_n$ 
over $\overline \Qq$, and finitely many suitably big primes $p$ are given with 
partitions $\{d_{p,1}, \ldots, d_{p,s_p}\}$ of $n$, there exist infinitely many specializations of $f$ at points 
$t_0\in \Qq$ that are degree $n$ field extensions with residue 
degrees $d_{p,1}, \ldots, d_{p,s_p}$ at each prescribed prime $p$.
%
The second one provides a description of 
the se\-pa\-ra\-ble closure of a PAC field $k$ of characteristic $p\not=2$: it is generated by all elements 
$y$ such that $y^m-y\in k$ for some $m\geq 2$.
The third one involves Hurwitz moduli spaces and concerns fields of definition of covers.
A common tool is a criterion for
an \'etale algebra $\prod_lE_l/k$ over a field $k$ to be the specialization 
of a $k$-cover $f:X\rightarrow B$ at some point $t_0\in B(k)$. The question is reduced 
to finding $k$-rational points on a certain $k$-variety, and then studied over the 
various fields $k$ of our applications. 

\end{abstract}

\maketitle



\section{Introduction}\label{sec:introduction}

The paper has three main applications which all are specialization results in Galois theory.

A first one is the following version of Hilbert's irreducibility theorem where a Grunwald like conclusion is conjoined with the usual irreducibility conclusion. 
\vskip 2mm

\noindent
{\bf Theorem 1} (corollary  \ref{cor:effective})
{\it Let $f:X\rightarrow \Pp^1$ be a degree $n$ $\Qq$-cover with geometric monodromy group\footnote{that is: the Galois group of the Galois closure of $f$ over $\overline{\Qq}$.} $S_n$  and $S$ be a finite set of primes $p$, suitably large (depending on $f$), each given with some positive integers $d_{p,1},\ldots,d_{p,s_p}$ of sum $n$. Then $f$ has infinitely many specializations that are degree $n$ field extensions of $\Qq$ with residue degrees $d_{p,1},\ldots,d_{p,s_p}$ at each
 $p\in S$.}
 \vskip 2mm

A second one is concerned with the arithmetic of PAC fields. Recall that a field $k$ is said to be PAC if every non-empty geometrically irreducible $k$-variety has a Zariski-dense set of $k$-rational points.
A typical example is the field $\Qq^\tr(\sqrt{-1})$ (which is also hilbertian and whose absolute Galois group is a free profinite group of countable rank); here $\Qq^\tr$ is the field of  totally real numbers (algebraic numbers such that all conjugates are real). See \cite{FrJa} for more  on PAC fields.

\vskip 2mm

\noindent
{\bf Theorem 2} (corollary \ref{cor:trinomial_realization})
{\it If $k$ is a PAC field of characteristic $p$, every extension $E/k$ of degree $n$ with $p\not| \hskip 2pt n(n-1)$ can be realized by a trinomial $Y^n - Y + b\in k[Y]$. Furthermore, if $p\not=2$, the separable closure $k^\sep$ is generated by all elements $y\in k^\sep$ such that $y^n-y\in k$ for some $n\geq 2$.  
 }
\vskip 2mm



We have similar applications about realizations by Morse polynomials (corollary \ref{cor:morse}), 
and
over finite fields (\S \ref{ssec:trinomials-variants}).



A third application concerns Hurwitz moduli spaces of covers of $\Pp^1$ with fixed branch point number $r$ and fixed monodromy group. Let $\hbox{\sf{H}}$ be a geometrically irreducible component of some Hurwitz space defined over some field $k$  and $N$ be the degree of the field of definition of the generic cover in $\hbox{\sf{H}}$ over that of its branch point divisor;
more formally $N$ is the degree of the natural cover $\hbox{\sf{H}} \rightarrow \hbox{\sf{U}}_r$ with $\hbox{\sf{U}}_r$ the configuration space for finite subsets of $\Pp^1$ of cardinality $r$ (see \S \ref{ssec:hurwitz_spaces}). We also make this assumption which can be checked in practice: the Hurwitz braid action restricted to $\hbox{\sf{H}}$ generates all of $S_N$ (more formally $S_N$ is the geometric monodromy group of the cover $\hbox{\sf{H}} \rightarrow \hbox{\sf{U}}_r$).  



\vskip 2mm

\noindent
{\bf Theorem 3} (corollary  \ref{cor:hurwitz_spaces})
{\it 
Consider the subset ${\mathcal U}\subset \hbox{\sf{U}}_r(k)$ of all ${\bf t}_0$ such that the $\overline k$-covers $f:X\rightarrow \Pp^1$ 
in $\hbox{\sf{H}}$ with branch divisor 
${\bf t}_0$ satisfy the following condition 
{\rm (in each case)}:
\vskip 1mm

\noindent
{\rm (a)} {\it {\rm(case $k$ PAC of characteristic $0$)}: their smallest fields of definition are given 
finite extensions $E_l/k$ ($l=1,\ldots,s$) with $\sum_{l=1}^s [E_l:k] = N$,

\vskip 1mm

\noindent
{\rm (b)} {\rm (case $k$ a number field)}: 

\noindent
- their fields of moduli are degree $N$ extensions of $k$, and

\noindent
- for each $v$ in a given finite set of finite places of $k$ with suitably big re\-si\-due field and residue characteristic (depending on $\hbox{\sf{H}}$), 
and each associated partition $\{d_{v,1}\ldots, d_{v,s_v}\}$ of $N$, the smallest definition fields of the covers $f\otimes_{\overline k}\overline{k_v}$ are the unramifed extensions of $k_v$ of degree $d_{v,1}, \ldots, d_{v,s_v}$. 
%

\vskip 1,6mm

\noindent  
Then {\rm (in each case)} ${\mathcal U}$ is a Zariski-dense subset of $\hbox{\sf{U}}_r(k)$. 
  }}

 \vskip 2mm

\noindent

\noindent
We refer to \S \ref{sec:applications} for more detailed statements and further applications.
\vskip 1mm

A common tool for these applications is a general specialization result.
If $f:X\rightarrow B$ is an algebraic cover defined over a field $k$ and 
$t_0$ a $k$-rational point on $B$, not in the finite list of branch points of $f$, the specialization of $f$ at $t_0$ is defined as a finite $k$-\'etale algebra of degree $n=\deg(f)$. For example, if $B=\Pp^1$ and $f$ is given by some
polynomial $P(T,Y)\in k[T,Y]$, it is the product of separable extensions of $k$ that correspond to the 
irreducible factors of $P(t_0,Y)$ (for all but finitely many $t_0\in k$). The central question is whether
a given $k$-\'etale algebra of degree $n$ is the specialization at some unramified point $t_0\in B(k)$ of some given degree $n$ $k$-cover $f:X\rightarrow B$. The classical Hilbert specialization property
corresponds to the special case \'etale algebras are taken to be single field extensions of 
degree $n$ and the answer is positive for at least one of them.

The question was investigated in \cite{DeBB} and \cite{DEGha} for Galois covers and some answers given that relate to the Regular Inverse Galois Problem (RIGP) and the Grunwald problem. 
We consider here the situation of not necessarily Galois covers. Our main tool 
is  a {\it twisting lemma} which extends the twisting lemma from the previous papers and 
gives a general answer to our question, under some hypothesis. The answer is that there exists
a certain ``twisted'' cover $\widetilde g: \widetilde Z \rightarrow B$ such that the \'etale algebra
{\it is} a specialization of the cover at some point $t_0\in B(k)$ if there exist  unramified $k$-rational points 
on $\widetilde Z$ (lemma \ref{prop:twisted cover}). The hypothesis is that the geometric monodromy group of the cover $f$ is the symetric group $S_n$ where $n=\deg(f)$; it is satisfied in many practical situations.

In \S \ref{sec:varying_field} we investigate the remaining problem of finding rational points
on $\widetilde Z$ over various fields: PAC fields, finite fields, complete fields, 
number fields. Corollaries \ref{cor:PAC} - \ref{cor:loc-glob} 
are answers to the original question in these situations. 
Applications are finally proved in \S \ref{sec:applications}. 



\section{The twisting lemma} \label{sec:local_result}

\subsection{Basic notation}  \label{ssec:basic_notation}  
Given a field $k$, fix an algebraic closure $\overline k$ and denote the separable closure of $k$ in $\overline k$ by $k^\sep$ and its absolute Galois group by $\Gabs_k$. If $k^\prime$ is an overfield of $k$, we use the notation $\otimes_kk^\prime$ for the scalar extension from $k$ to $k^\prime$: for example, if $X$ is a $k$-curve,
$X\otimes_kk^\prime$ is the $k^\prime$-curve obtained by scalar extension. For more on this subsection, we refer to \cite[\S 2]{DeDo1} or \cite[chapitre 3]{coursM2}.

\subsubsection{Etale algebras and their Galois representations} \label{sssec:etale_algrebras}
Given a field $k$, a {\it $k$-\'etale algebra} is a product $\prod_{l=1}^s E_l/k$ of $k$-isomorphism classes of finite sub-field extensions $E_1/k, \ldots, E_s/k$ of $k^\sep/k$. Set $m_l=[E_l:k]$, $l=1,\ldots,s$ and $m=\sum_{l=1}^s m_l$. If $N/k$ is a Galois extension containing the Galois closures of $E_1/k, \ldots, E_s/k$, the Galois group $\Gal(N/k)$ acts by left multiplication on the left cosets of $\Gal(N/k)$ modulo $\Gal(N/E_l)$ for each $l=1,\ldots,s$. The resulting action $\Gal(N/k) \rightarrow S_m$ on all these left cosets, which is well-defined up to conjugation by elements of $S_m$, is called the {\it Galois representation of $\prod_{l=1}^s E_l/k$ relative to $N$}. Equivalently it can be defined as the 
action of $\Gal(N/k)$ on the set of all $k$-embeddings $E_l \hookrightarrow N$, $l=1,\ldots,s$.  

Conversely, an action $\mu: \Gal(N/k) \rightarrow S_m$ determines a $k$-\'etale algebra in the following way. For $i=1,\ldots,m$, denote the fixed field in $N$ of the subgroup of $\Gal(N/k)$ consisting of all $\tau$ such that $\mu(\tau)(i) = i$ by $E_i$. 
The product $\prod_lE_l/k$ for $l$ ranging over a set of representatives of the orbits 
of the action $\mu$ and where each extension $E_l/k$ is regarded modulo $k$-isomorphism is a $k$-\'etale algebra with $\sum_l [E_l:k] = m$. 

\vskip 1,5mm

\noindent
{\it {\rm G}-Galois variant}: if $\prod_{l=1}^s E_l/k$ is a {\it single Galois extension} $E/k$, the restriction
$\Gal(N/k)\rightarrow \Gal(E/k)$ is called {\it the {\rm G}-Galois representation of $E/k$} (relative to $N$). Any map $\varphi: \Gal(N/k) \rightarrow G$ obtained by composing $\Gal(N/k)\rightarrow \Gal(E/k)$ with a monomorphism $\Gal(E/k) \rightarrow G$ is called a G-{Galois} representation of $E/k$ 
(relative to $N$).
The extension $E/k$ can be recovered 
from $\varphi: \Gal(N/k) \rightarrow G$ by taking the fixed field in $N$ of ${\rm ker}(\varphi)$.  One obtains the Galois representation $\Gal(N/k)\rightarrow S_n$ of $E/k$ (relative to $N$) from a G-Galois representation $\varphi: \Gal(N/k) \rightarrow G$ (relative to $N$) by composing it
with the left-regular representation of the image group $\varphi(\Gal(N/k))$.

\subsubsection{Covers and function field extensions} Given a regular projective geometrically irreducible $k$-variety $B$, a {\it $k$-mere cover of $B$} is a finite and generically unramified morphism $f:X \rightarrow B$ defined over $k$ with $X$ a normal and
geometrically irreducible variety. 
Through the function field functor $k$-mere covers $f:X \rightarrow B$ correspond to finite separable field extensions $k(X)/k(B)$ that are regular over $k$ ({\it i.e. } $k(X)\cap \overline k=k$). 
The Galois group of the Galois closure $\widehat{k(X)}/k(B)$ of  $k(X)/k(B)$ is called the {\it monodromy group} of $f$ and the monodromy group of the $k^\sep$-mere cover $f\otimes_k k^\sep$ the {\it geometric monodromy group}. The Galois closure $\widehat{k(X)}/k(B)$ need not be a regular extension of $k$; it does if and only if the monodromy group and the geometric monodromy group coincide. This happens for example if $\widehat{k(X)}=k(X)$ ({\it i.e.} if $f$ is Galois), or if $f$ is of degree $n$ and geometric monodromy group $S_n$. When the Galois closure $\widehat{k(X)}/k(B)$ is regular over $k$, it  corresponds to a Galois $k$-mere cover $g:Z\rightarrow B$ called the Galois closure of $f$.

The term ``mere'' used above is meant to distinguish mere covers from G-covers. By {\it$k$-{\rm G}-cover of $B$ of group $G$}, we mean a Galois cover $f:X\rightarrow B$ over $k$ given together with an
isomorphism $G\rightarrow \Gal(k(X)/k(B))$. Viewed as a mere cover ({\it i.e.} without the isomorphism $G\rightarrow \Gal(k(X)/k(B))$), $f:X\rightarrow B$ is a Galois $k$-mere cover. G-covers of $B$ of group $G$ over $k$ correspond to regular Galois extensions $k(X)/k(B)$ given with 
an isomorphism of the Galois group $\Gal(k(X)/k(B))$ with $G$.  

We sometimes abuse terminology and call G-covers (resp. mere covers) {\it regular Galois covers} (resp. {\it regular covers}).


By {\it branch divisor} of a $k$-cover $f$ (mere or G-), we mean that of the $k^\sep$-cover  $f\otimes_kk^\sep$, {\it i.e.} the formal sum of all hypersurfaces of $B$ such that the associated discrete valuations are ramified in the extension $k^\sep(X)/k^\sep(B)$. 


\subsubsection{$\pi_1$-representations}
Given a reduced positive divisor $D\subset B$, denote the {\it $k$-fundamental group} of $B\setminus D$ by $\pi_1(B\setminus D, t)_k$ where $t\in B(\overline k)\setminus D$ is a base point. Conjoining the two dictionaries covers-function field extensions and field extensions-Galois representations, we obtain the following correspondences.

Mere covers of $B$ of degree $n$ (resp. G-covers of $B$ of group $G$) with branch divisor contained in $D$ correspond to transitive morphisms $\pi_1(B\setminus D, t)_k \rightarrow S_n$ such that the restriction to 
$\pi_1(B\setminus D, t)_{k^\sep}$ is transitive (resp. to epimorphisms $\pi_1(B\setminus D, t)_k \rightarrow G$ such that the restriction to $\pi_1(B\setminus D, t)_{k^\sep}$ is onto). 
These morphisms are called {\it fundamental group representations} ($\pi_1$-representations for short)
of the corresponding $k$-covers (mere or G-).

\subsubsection{Specializations} \label{ssec:specialization}
Each $k$-rational point $t_0\in B(k)\setminus D$ provides a section
${\rm s}_{t_0}: \Gabs_k\rightarrow \pi_1(B\setminus D, t)_k$ to the exact sequence

$$ 1\rightarrow \pi_1(B\setminus D, t)_{k^\sep} \rightarrow \pi_1(B\setminus D, t)_k \rightarrow \Gabs_k \rightarrow 1$$ 

\noindent
well-defined up to conjugation by elements in $\pi_1(B\setminus D, t)_{k^\sep}$. 

If $\phi: \pi_1(B\setminus D, t)_k \rightarrow G$ represents a $k$-G-cover $f:X\rightarrow B$, 
the morphism $\phi \circ {\sf s}_{t_0}:\Gabs_k \rightarrow G$ is a G-Galois representation. The fixed field in $k^\sep$ of ${\rm ker}(\phi \circ {\sf s}_{t_0})$ is the residue field at some/any point above $t_0$
in the extension $k(X)/k(B)$. We denote it by $k(X)_{t_0}$ and call $k(X)_{t_0}/k$ {\it the specialization} of the $k$-G-cover $f$ at $t_0$.

If $\phi: \pi_1(B\setminus D, t)_k \rightarrow S_n$ represents a $k$-mere cover $f:X\rightarrow B$, the morphism $\phi \circ {\sf s}_{t_0}:\Gabs_k \rightarrow S_n$ is the {\it specialization representation} 
of $f$ at $t_0$. The corresponding $k$-\'etale algebra is denoted by 
$\prod_{l=1}^s k(X)_{t_0,l}/k$ and called the {\it collection of specializations} of $f$ at $t_0$. Each field $k(X)_{t_0,l}$ is a residue extension at some prime above $t_0$ in the extension $k(X)/k(B)$ and {\it vice-versa}; $k(X)_{t_0,l}$ is called {\it a specialization} of $f$ at $t_0$. Geometrically the fields $k(X)_{t_0,l}$ correspond to the definition fields of the points in the fiber $f^{-1}(t_0)$ and $\phi \circ {\sf s}_{t_0}:\Gabs_k \rightarrow S_n$ to the {\it action} of $\Gabs_k$ on these points. The {\it compositum} in $k^\sep$ of the Galois closures of all spe\-cia\-li\-zations at $t_0$ is {\it the} specialization at $t_0$ of the Galois closure of $f$.



\subsection{The twisting lemma} \label{sec:twisting}

Let $k$ be a field, $f:X\rightarrow B$ be a $k$-mere cover 
and $\prod_{l=1}^s E_l/k$ be a $k$-\'etale algebra. 
The question we address is whether $\prod_{l=1}^s E_l/k$ {is the collection $\prod_{l} k(X)_{t_0,l}/k$ of specializations of $f:X\rightarrow B$} at some unramified point $t_0\in B(k)$. Lemma \ref{prop:twisted cover} gives a sufficient condition for the answer to be affirmative.

\subsubsection{Statement of the twisting lemma} \label{ssec:twisting-statement} 
We assume that $f:X\rightarrow B$ is of degree $n$ and geometric monodromy group $S_n$.
Denote the Galois closure of the cover $f$ by $g: Z \rightarrow B$, the {\it compositum} inside $k^\sep$ of the Galois closures of the extensions $E_l/k$, $l=1,\ldots, s$ by $N/k$. The {\it twisted cover} $\widetilde {g}^{N}: \widetilde {Z}^{N} \rightarrow B$ in the statement below is a $k$-mere cover obtained by twisting the Galois $k$-mere cover $g: Z \rightarrow B$ by the Galois extension $N/k$. Its precise definition is given in \cite{DEGha} and is recalled in  \S \ref{ssec:twisting-lemma-proof}; it is in particular a $k$-model of $g\otimes_kk^\sep$.

\begin{twisting lemma} \label{prop:twisted cover} 
Assume $f:X\rightarrow B$ is a degree $n$ $k$-mere cover with geometric monodromy group $S_n$ and $\prod_{l=1}^s E_l/k$ is a $k$-\'etale algebra with $\sum_{l=1}^s [E_l:k] = n$. Then the twisted cover $\widetilde {g}^{N}: \widetilde {Z}^{N} \rightarrow B$ has the following property. For each unramified point $t_0 \in B(k)$, \vskip 1mm

\noindent
\hskip 5,5mm {\it if} \hskip 1mm {\rm (i)}  there exists a point $x_0\in \widetilde Z^{N}(k)$ such that $\widetilde g^{N}(x_0)=t_0$, 
\vskip 1mm

\noindent
{\it then} \hskip 0,5mm {\rm (ii)} $\prod_{l} E_l/k$ is the collection $\prod_{l} k(X)_{t_0,l}/k$ of specializations of 
$f$ 

\hskip 11mm at the point $t_0$.
\end{twisting lemma}

For $B=\Pp^1$, a polynomial form of the statement can be given for which the $k$-mere cover is replaced by a polynomial $P(T,Y) \in k[T,Y]$ of degree $n$ and with Galois group $S_n$ over $\overline k$, as a polynomial in $Y$. For all but finitely many $t_0\in k$, implication {\rm (i)}  $\Rightarrow$ {\rm (ii)} holds  with condition {\rm (ii)} translated as follows:
\smallskip

\noindent
{\rm (ii)}  {\it the polynomial $P(t_0,Y)$ factors as a product  $\prod_{l=1}^s Q_l(Y)$ of polynomials $Q_l$ irreducible in $k[Y]$ and such that $E_l/k$ is generated by one of its roots, $l=1,\ldots,s$.}
\vskip 1,5mm

With some adjustments, some converse (ii) $\Rightarrow$ (i) also holds in the twisting lemma \ref{prop:twisted cover}. It is also possible to relax the assumption that the geometric monodromy group of $f:X\rightarrow B$ is $S_n$, at the cost of some technical complications. This is explained in \cite{DeLe2}, together with some specific applications.

\subsubsection{Proof of the twisting lemma} \label{ssec:twisting-lemma-proof} 
Let $H=\Gal(N/k)$, $\varphi: \Gabs_k \rightarrow H$ be the G-Galois representation 
of $N/k$ relative to $k^\sep$
 and $\mu: H\rightarrow S_n$ be the Galois representation of 
$\prod_{l=1}^s E_l/k$ relative to $N$. The map $\mu\circ \varphi: \Gabs_k \rightarrow
S_n$ is then the Galois representation of $\prod_{l=1}^s E_l/k$ relative to $k^\sep$.
Denote the $s$ orbits of $\mu: H\rightarrow S_n$, which are the same as the orbits
of $\mu\circ \varphi: \Gabs_k \rightarrow S_n$, by ${\mathcal O}_1,\ldots, {\mathcal O}_s$;
they correspond to the extensions $E_1,\ldots E_s$. Fix one of these orbits, {\it i.e.} 
$l\in \{1,\ldots,s\}$, and let $i\in \{1,\ldots, n\}$ be 
some index such that $E_l$ is the fixed field in $k^{\sep}$ of the subgroup 
of $\Gabs_k$ fixing $i$ {\it via} the action $\mu\circ \varphi$. 

As the $k$-mere cover $f:X\rightarrow B$ is of degree $n$ and that the Galois group 
$\Gal(k^\sep(Z)/k^\sep(B))$ is assumed to be isomorphic to $S_n$, the same is true of  
$\Gal(k(Z)/k(B))$. Therefore $k(Z)$ is a regular extension of $k$, or, in other words, 
$g:Z\rightarrow B$ is a $k$-G-cover. Let $\phi:  \pi_1(B\setminus D, t)_k \rightarrow S_n$ 
be the corresponding $\pi_1$-representation (where $D$ is the branch divisor of $f$).

With ${\rm Per}(S_n)$ the permutation group of $S_n$, consider the map 

$$\widetilde \phi^{\mu\varphi}: \pi_1(B\setminus D, t)_k \rightarrow {\rm Per}(S_n)$$

\noindent
defined by this formula, where 
$r$ is the restriction $\pi_1(B\setminus D, t)_k \rightarrow \Gabs_k$:
for  $\theta\in \pi_1(B\setminus D, t)_k$ and $x\in S_n$,

$$ \widetilde \phi^{\mu\varphi}(\theta)(x)  =
\phi(\theta) \hskip 3pt x  \hskip 4pt (\mu \circ \varphi \circ r) (\theta)^{-1}$$

\noindent
It is easily checked that $\widetilde \phi^{\mu\varphi}$ is a group homomorphism,
with the same restriction on $\pi_1(B\setminus D, t)_{k^\sep}$ as $\phi$ (composed with the left-regular representation of $S_n$). Hence the corresponding action is transitive. Denote the corresponding $k$-mere cover by $\widetilde g^{N}: \widetilde Z^{N} \rightarrow B$ 
and call it  the {\it twisted cover} of $g$ by the extension $N/k$; it is a $k$-model of the $k^\sep$-mere cover $g\otimes_k{k^\sep}$.
The twisted cover $\widetilde g^{N}: \widetilde Z^{N} \rightarrow B$ was defined in \cite{DEGha} (and originally in \cite{DeBB}) where is also given its main property that we are using below.

Let $t_0\in B(k)\setminus D$ and assume that condition (i) from lemma \ref{prop:twisted cover} holds, {\it i.e.}, there exists $x_0\in \widetilde Z^{N} (k)$ such that $\widetilde g^{N} (x_0)=t_0$. Then from \cite[lemma 2.1]{DEGha}, there exists $\omega \in S_n$ such that 
\vskip 3mm

\centerline{$\phi ({\sf s}_{t_0}(\tau)) = \omega \hskip 3pt (\mu \circ \varphi) (\tau) \hskip 3pt \omega^{-1}\hskip 4mm (\tau \in \Gabs_k)$}

\vskip 2mm
%

%


\noindent
where ${\rm s}_{t_0}: \Gabs_k\rightarrow \pi_1(B\setminus D, t)_k$ is the section associated with $t_0$ (\S \ref{ssec:specialization}). It follows that for $j=\omega(i)$, we have, for every $\tau \in \Gabs_k$,

$$\phi ({\sf s}_{t_0}(\tau)) (j) = \omega \hskip 2pt (\mu \circ \varphi)(\tau) \hskip 2pt (i)$$

\noindent
and so $j$ is fixed by $\phi ({\sf s}_{t_0}(\tau))$ if and only if $i$ is fixed by $(\mu \circ \varphi)(\tau)$. Conclude that the specialization $k(X)_{t_0,j}$ and the field $E_l$ coincide.  $\square$


\penalty -3000
\section{Varying the base field} \label{sec:varying_field}

We consider the general problem over various base fields $k$. We start with the 
case of PAC fields (\S \ref {ssec:field_PAC}), which after a first result in \cite{DeBB}, has also been studied in parallel by Bary-Soroker; see \cite{Bary-Soroker_irreducible}, 
\cite[corollary 1.4]{Bary-Soroker_Dirichlet}. \S \ref{ssec:finite_fields} is devoted to finite fields 
for which various forms of the results also exist in the literature. These two cases are presented 
here as 
special cases of our unifying approach.
\S \ref{ssec:local} and \S \ref{ssec:field_loc-glob} give newer applications, to the cases $k$ is a complete field and $k$ is a number field.

\subsection{PAC fields} \label{ssec:field_PAC}
If $k$ is a PAC field, then condition (i) from lemma \ref{prop:twisted cover} holds for all $t_0$ in a Zariski dense subset of $B(k)\setminus D$; consequently so does condition (ii). 

\begin{corollary} \label{cor:PAC}
Let $k$ be a PAC field and $f:X\rightarrow B$ be a $k$-mere cover of degree $n$ and geometric monodromy group $S_n$. If $\prod_{l=1}^s E_l/k$ is a $k$-\'etale algebra with $\sum_{l=1}^s [E_l:k] = n$, then for all $t_0$ in a Zariski dense subset of $B(k)\setminus D$, the collection $\prod_{l} k(X)_{t_0,l}/k$ of specializations of $f$ at $t_0$ is $\prod_{l} E_l/k$.
\end{corollary}


A similar result is \cite[theorem 2.4]{Bary-Soroker_irreducible}.
As a special case we obtain this statement, which is also  \cite[corollary 1.4]{Bary-Soroker_Dirichlet}: if $P(T,Y) \in k[T,Y]$ is a polynomial of degree $n$ and Galois group $S_n$ over $\overline k$, as a polynomial in $Y$ and $E/k$ is a degree $n$ separable extension, then there exist infinitely many $t_0 \in k$ such that $P(t_0,Y)$ is irreducible in $k[Y]$ and has a root in $\overline k$ that generates $E/k$.

\subsection{Finite fields} \label{ssec:finite_fields}
Assume $k=\Ff_q$ is the finite field of order $q$ and as above consider the case of covers of $\Pp^1$ (for simplicity). From the Lang-Weil estimates for the number of rational points on a curve over $\Ff_q$, condition (i) from lemma \ref{prop:twisted cover} holds for at least one unramified $t_0\in k$ if $q+1 -2\widetilde {\rm g} \sqrt{q} > \widetilde rn!$ where $\widetilde {\rm g}$ is the genus of the covering space $\widetilde Z^{N}$ of the mere cover $\widetilde g^{N}$ from lemma \ref{prop:twisted cover} 
and $\widetilde r$ the branch point number of $\widetilde g^{N}$. But $\widetilde g^{N}\otimes_kk^\sep \simeq g\otimes_kk^\sep$ (where $g:Z\rightarrow \Pp^1$ is as before the Galois closure of $f:X\rightarrow \Pp^1$). Consequently $\widetilde r$ is merely the branch point number of $f$ and $\widetilde {\rm g}$ the genus of $g:Z\rightarrow \Pp^1$. Using the Riemann-Hurwitz formula, it is readily checked that $q \geq 4r^2 (n!)^2$ suffices to guarantee the preceding inequality. We obtain the following. 

\begin{corollary} \label{cor:finite-fields}
Let $f:X\rightarrow \Pp^1$ be a $\Ff_q$-mere cover of degree $n$, with $r$ branch points and with geometric monodromy group $S_n$. Assume that $q \geq 4r^2 (n!)^2$. Then for every positive integers $d_1,\ldots,d_s$ (possibly repeated) such that $\sum_{l=1}^s d_l = n$, there exists at least one $t_0 \in \Ff_q$ such that $\prod_{l=1}^s\Ff_{q^{d_l}}/\Ff_q$ is the collection of specializations of $f$ at $t_0$. 
\end{corollary}

One can even evaluate the number of $t_0 \in \Ff_q$ for which the conclusion holds: 
for example for $s=1$ and $d_1=n$, it is of the form $q/n + O(\sqrt{q})$. See \cite{DeLe2} for details
on this extra conclusion (which uses the converse (ii) $\Rightarrow$ (i) in the twisting lemma 
alluded to in \S \ref{ssec:twisting-statement}). 

\subsection{Complete valued fields} \label{ssec:local}
Assume $k$ is the quotient field of some complete discrete valuation ring $A$. 
Denote the valuation ideal by $\mathfrak{p}$, the residue field  
by $\kappa$, assumed to be perfect, and its characteristic by $p\geq 0$. A $k$-\'etale algebra $\prod_{l=1}^s E_l/k$ is said to be {\it unramified} if each field extension $E_l/k$ is unramified.

Let $B$ be a smooth projective and geometrically irreducible $k$-variety given
with an integral smooth projective model $\mathcal{B}$ over $A$.
Let  $f:X\rightarrow B$ be a degree $n$ $k$-mere cover with branch divisor $D$. Denote the Zariski closure of $D$ in ${\mathcal B}$ by ${\mathcal D}$,  the normalization of ${\mathcal B}$ in $k(X)$ by ${\mathcal F}: {\mathcal X}\rightarrow {\mathcal B}$ and its special fiber by ${\mathcal F}_\spe: {\mathcal X}_\spe \rightarrow {\mathcal B}_\spe$.

The constant $c(f,{\mathcal B})$ in the statement below only depends on $f$ and ${\mathcal B}$. It is the constant $c(g,{\mathcal B})$ from \cite{DEGha} for $g:Z\rightarrow B$ the Galois closure of $f:X\rightarrow B$. For $B=\Pp^1$, it can be taken to be $c(f,{\mathcal B}) =4r^2 (n!)^2$ where $r$ is the branch point number of $f$. See \cite[\S 2.5]{DEGha} for a more general description.


\begin{corollary} \label{cor:loc}
Let $k$, $\mathcal{B}$ and $f:X\rightarrow B$ be as above and $\prod_{l=1}^s E_l/k$ be an unramified $k$-\'etale algebra with \hbox{$\sum_{l=1}^s [E_l:k] = n$}. Assume that the geometric monodromy group of $f:X\rightarrow B$ is $S_n$ and that these two further conditions hold:

\vskip 1,5mm

\noindent
{\rm (good-red)} $p=0$ or $p>n$, ${\mathcal D}$ is smooth, ${\mathcal D}\cup {\mathcal B}_\spe$ is 
regular with normal crossings over $A$, and there is no vertical ramification at ${\mathfrak p}$ in the Galois closure $g:Z\rightarrow B$\footnote{
see \cite{DEGha} for a precise definition of non-vertical ramification. This condition can in fact be removed here if $n\geq 3$: according to a lemma of Beckmann \cite{Beckmann}, no vertical ramification may then occur (under the other assumptions $p>n$ and ${\mathcal D}$ \'etale) as the geometric monodromy group $S_n$ is of trivial center.}.
\vskip 1,5mm

\noindent
{\rm ($\kappa$-big-enough)} $\kappa$ is a PAC field or is a finite field of order 
$q\geq c(f,{\mathcal B})$.
\vskip 1,5mm

\noindent
Then there exist points $t_0\in B(k)\setminus D$ such that  $\prod_{l=1}^s E_l/k$ is the collection of specializations of $f$ at $t_0$. More precisely, the set of such points $t_0$ contains the preimage {\it via} the map ${\mathcal B}(A) \rightarrow {\mathcal B}_\spe(\kappa)$ of a non-empty subset $F\subset {\mathcal B}_\spe(\kappa) \setminus {\mathcal D}_\spe$.
 \end{corollary}

\vskip 2mm

\begin{proof} Let $\widetilde g^{N}: \widetilde Z^{N} \rightarrow B$ be the $k$-mere cover from 
lemma \ref{prop:twisted cover}, obtained by twisting the Galois closure $g:Z\rightarrow B$ of $f:X\rightarrow B$ by the {\it compositum} $N/k$ of the Galois closures of $E_1/k, \ldots, E_s/k$. From lemma \ref{prop:twisted cover}, it 
suffices to show that $\widetilde Z^{N}$ has $k$-rational points.  This 
(and the more precise conclusion of corollary \ref{cor:loc}) is explained in proposition 2.2 
and lemma 2.4 from \cite{DEGha} which we summarize below.

Denote by ${\mathcal G}: {\mathcal Z}\rightarrow {\mathcal B}$ the normalization of ${\mathcal B}$ in $k(Z)$. Assumption {\rm (good-red)} holds for ${\mathcal G}$ as it holds for ${\mathcal F}$ ($f$ and $g$ have the same branch divisor) 
and the Galois extension $N/k$ is unramified ({\it compositum} of 
unramified extensions). These two conditions guarantee that the morphism $\widetilde {\mathcal G}^{N}: \widetilde {\mathcal Z}^{N} \rightarrow {\mathcal B}$ obtained by normalizing ${\mathcal B}$ in $k(\widetilde Z^{N})$ has good reduction (including no vertical ramification at ${\mathfrak p}$) \cite[proposition 2.2]{DEGha}. Assumption {\rm ($\kappa$-big-enough)} shows next that $\kappa$-rational points exist on the special fiber $\widetilde {\mathcal Z}^{N}_\spe$; if $\kappa$ is finite, this follows from the Lang-Weil estimates \cite[lemma 2.4]{DEGha}. Hensel's lemma is finally used to lift these $\kappa$-rational points to $k$-points on $\widetilde Z^{N}$. \end{proof}

\subsection{Local-global results} \label{ssec:field_loc-glob} 
Finding rational points on varieties over a global field $k$ is harder than it is over local fields. Nevertheless results from \S \ref{ssec:local} can be used to obtain local-global statements. We explain below how to globalize local information coming from corollary \ref{cor:loc}.
  
Let  $k$ be the quotient field of some De\-de\-kind domain $R$ and $S$ be a finite set of places of $k$ corresponding to some prime ideals in $R$. For every place $v$,  the completion of $k$ is denoted by $k_v$, the valuation ring by $R_v$, the valuation ideal by ${\frak p}_v$, the residue field  by $\kappa_v$ which we assume to be perfect, the order (possibly infinite) of $\kappa_v$ by $q_v$ and its characteristic by $p_v$.  

Let $B$ be a smooth projective and geometrically integral $k$-variety,  given with an integral model ${\mathcal B}$ over $R$ such that ${\mathcal B}_v= {\mathcal B}\otimes_{R}R_v$ is smooth for each $v\in S$. The {\it weak approximation property} below guarantees that $k_v$-rational points on $B$ ($v\in S$) that may be provided by corollary \ref{cor:loc} can be approximated by some $k$-rational point on $B$. 
\vskip 2,5mm

\noindent
{\rm (weak-approx\hskip2pt$/S$)} \hskip 2mm  {\it $B(k)$ is dense in $\prod_{v\in S} B(k_v)$.} 
\vskip 2mm

\noindent
The next statement then readily follows from corollary \ref{cor:loc}.


\begin{corollary} \label{cor:loc-glob}
Let $k$, $S$, $\mathcal{B}$ be as above, $f:X\rightarrow B$ be a degree $n$ $k$-mere cover with  branch divisor $D$, and, for each $v\in S$, let $\prod_{l=1}^{s_v} \hskip -1pt E_{v,l}/k_v$ be an unramified $k_v$-\'etale algebra \hbox{with $\sum_{l=1}^{s_v} [E_{v,l}:k_v] \hskip -1pt = \hskip -1pt n$}.
\vskip 1mm   

\noindent
Assume that

\noindent 
- the geometric monodromy group of $f:X\rightarrow B$ is $S_n$,

\noindent 
- the weak approximation condition  {\rm (weak-approx\hskip2pt$/S$)} holds, and

\noindent
- for each $v\in S$, assumptions {\rm (good-red)} and {\rm ($\kappa$-big-enough)} of co\-rol\-lary \ref{cor:loc}  hold for the $k_v$-mere cover $f\otimes_k k_v$ and the residue field $\kappa_v$. 
%

\noindent

\vskip 1mm   

\noindent 
Then there exist $v$-adic open subsets $U_v\subset B(k_v)\setminus D$ ($v\in S$) such that $B(k)\cap \prod_{v\in S} U_v\not= \emptyset$ and the following holds: for each $t_0 \in B(k)\cap \prod_{v\in S} U_v$ and each $v\in S$, the \'etale algebra $\prod_{l=1}^{s_v} E_{v,l}/k_v$ is the collection of specializations of $f\otimes_k k_v$ at $t_0$.

 \end{corollary}
 
\noindent
{\it Addendum \ref{cor:loc-glob}.}
Each condition $q_v \geq c(f\otimes_kk_v, {\mathcal B}\otimes_kk_v$) from assumption {\rm ($\kappa_v$-big-enough)} can be guaranteed by some condition $q_v \geq C(f,{\mathcal B})$ where $C(f,{\mathcal B})$  only depends on $f$ and ${\mathcal B}$ (and not on $v$). This constant is the constant $C(g,{\mathcal B})$ from \cite{DEGha} with $g$ the Galois closure of $f$. For $B=\Pp^1$, it can be taken to be $C(f,{\mathcal B}) = 4r^2 (n!)^2$ where $r$ is the branch point number of $f$. See \cite[\S 3.1]{DEGha} for a more general description.
\vskip 2mm

\section{Applications} \label{sec:applications}

The three subsections of this section correspond to the three main applications presented in
the introduction.

\subsection{Hilbert's irreducibility theorem} \label{ssec:HIT}

We elaborate on the local-global results from \S \ref{ssec:field_loc-glob} 
when $B=\Pp^1$. In this situation, assumption {\rm (weak-approx\hskip2pt$/S$)} holds for every $S$ and the good reduction assumption {\rm (good-red)} requires that each place $v\in S$ be {\it good}, by which we mean that $p_v=0$ or $p_v>n$, {\it the branch point set ${\bf t}=\{t_1,\ldots,t_r\}$ of $f$ is \'etale} ({\it i.e.}, no two distinct branch points {\it coalesce} modulo the valuation ideal of $v$) and 
there is no vertical ramification in the Galois closure of $f$ \footnote{This last condition is automatic if the geometric monodromy group is $S_n$ with $n=\deg(f)\geq 3$.}.

\subsubsection{A standard trick} \label{ssec:standard_trick}
Over certain fields ({\it e.g.} number fields), there is a trick that makes it possible, at the cost of throwing in more places in $S$, to further guarantee in corollary \ref{cor:loc-glob} that Hilbert's irreducibility conclusion holds, {\it i.e.} that the collection of specializations of $f$ at $t_0$ consists of a single field extension $k(X)_{t_0}/k$ of degree $n$.

Namely the idea is to construct a finite set $S_0$ of finite places of $k$, disjoint from $S$, and to attach to each $v\in S_0$ a $k_v$-\'etale algebra $\prod_{l} E_{v,l}/k_v$ with all $E_{v,l}/k_v$ trivial but one consisting of an unramified Galois extension $E_v/k_v$ of group $H_v$ of order $\leq n$. If the assumptions of corollary \ref{cor:loc-glob} still hold for the set $T=S\cup S_0$, then it follows from its conclusion that $\Gal(k(Z)_{t_0}/k)$ contains some conjugate subgroup $H_v^{g_v}$ of $H_v$ for some $g_v\in S_n$ ($v\in S_0$). In some situations, this last condition implies that $\Gal(k(Z)_{t_0}/k)$ is all of $S_n$. This is the case for example if $S_0$ contains $3$ places with corresponding $H_v$ cyclic subgroups respectively generated by a $n$-cycle, a $n-1$-cycle and a $2$-cycle. Of course, for this idea to work, Galois extensions $E_v/k_v$ with groups $H_v$ should exist, for places $v$ satisfying the assumptions of corollary \ref{cor:loc-glob}. 

\subsubsection{The number field case} We develop the number field case for which this trick can be used. Another example would be to work over $k=\kappa(x)$ with $\kappa$ a PAC field with enough cyclic extensions.
We will also use the explicit aspect of \cite{DEGha} that makes it possible to be more precise on the constants. For simplicity take $k=\Qq$. The next statement is the outcome of the above considerations, together with corollary \ref{cor:loc-glob} (a detailed proof is easily obtained by adjusting the proof of \cite[corollary 4.1]{DEGha}).

\begin{corollary} \label{cor:effective} 
Let  $f:X\rightarrow \Pp^1$ be a degree $n$ $\Qq$-mere cover with geometric monodromy group $S_n$. There exist integers 
$m_0, \beta>0$ depending on $f$ with the following property.
%
Let ${\mathcal S}$ be a finite set of good primes $p>m_0$, each given with positive integers $d_{p,1}\ldots, d_{p,s_p}$ (possibly repeated) with $\sum_{l=1}^{s_p} d_{p,l}= n$. Then there exists $b \in \Zz$ such that 
\vskip 1,2mm

\noindent
{\rm (i)} $0\leq b \leq \beta \prod_{p\in {\mathcal S}} p$,  

\vskip 1,2mm

\noindent
{\rm (ii)} for each integer $t_0\equiv b$ {\rm mod} $(\beta \prod_{p\in {\mathcal S}} p)$, $t_0$ is not a branch point of $f$ and the collection of specializations of $f$ at $t_0$ consists of a single field extension $\Qq(X)_{t_0}/\Qq$ of degree $n$ which has residue degrees $d_{p,1}\ldots, d_{p,s_p}$ at $p$ for each $p\in {\mathcal S}$, and $S_n$ as Galois group of its Galois closure.


\end{corollary}

\noindent
{\it Addendum} \ref{cor:effective} (on the constants) 
Denote 
the number of branch points of $f$ 
by $r$ and the number of bad primes by ${\hbox{\rm br}}({\bf t})$. One can take $m_0$ such 
that the interval $[4r^2 (n!)^2,m_0]$ contains at least ${\hbox{\rm br}}({\bf t})+3$ distinct primes, and $\beta$ to be the product of $3$ good 
primes in $[4r^2 (n!)^2,m_0]$. 
\medskip

If the cover $f:X\rightarrow \Pp^1$ is given by a polynomial $P(T,Y)\in \Qq[T,Y]$, {addendum}  \ref{cor:effective} provides a bound for the least specialization $t\geq 0$ making $P(t,Y)$  irreducible in $\Qq[Y]$ that depends only on $\deg_Y(P)$, $r$ and ${\hbox{\rm br}}({\bf t})$. It is conjectured that a bound depending only on $\deg(P)$ exists in general for Hilbert's irreducibility theorem (see \cite{DeWa}).


\subsection{Trinomial realizations and variants}  

Bary-Soroker's motivation in \cite{Bary-Soroker_Dirichlet} was to obtain analogs of Dirichlet's theorem
for polynomial rings. He proved that if $k$ is a PAC field, then given $a(Y), b(Y) \in k[Y]$ relatively prime, for every integer $n$, suitably large (depending on $a(Y), b(Y)$) and for which $k$ has at least one degree $n$ separable extension, there are infinitely many $c(Y)\in k[Y]$ such that $a(Y)+b(Y)c(Y)\in k[Y]$ is irreducible and of degree $n$. The strategy is to construct $c_0(Y)\in k[Y]$ such that $a(Y)+b(Y)c_0(Y) T\in k[T,Y]$ is absolutely irreducible, of degree $n$ and Galois group $S_n$ over $\overline k(T)$, and then to specialize $t$ properly (as in \S \ref{ssec:field_PAC}).
We develop below other applications. 

\subsubsection{Classical regular realizations of $S_n$} \label{ssec:classical_realizations} An hypothesis in our results from \S \ref{sec:varying_field} is that the $k$-mere cover $f:X\rightarrow B$ is of degree $n$ and
of geometric monodromy group $S_n$. We recall below some classical covers 
$f:X\rightarrow \Pp^1$ with these properties. The cover $f$ is given by a polynomial
$P(T,Y)\in k[T,Y]$, the map $f$ corresponding to the $T$-projection $(t,y)\rightarrow t$ from the curve $P(t,y)=0$ to the line.
Fix the integer $n\geq 2$.
\vskip 2mm

\noindent
(a) ({\it Trinomials}):  
$k$ is a field and $P(T,Y)=  Y^n-T^r Y^m + T^s \in k[T,Y]$ where $n, m, r, s$ are positive integers such that $1\leq m< n$,  $(m,n)=1$, the characteristic $p\geq 0$ of $k$ does not divide $mn(m-n)$ and $s(n-m)-rn=1$. The branch points of the associated cover $f:X\rightarrow \Pp^1$ are $0$, $\infty$ and $t_0=m^m n^{-n} (n-m)^{n-m}$ with corresponding ramification indices $m(n-m)$ at $0$, $n$ at $\infty$ and $2$ at $t_0$. 
See \cite[\S 2.4]{Schinzel-book}. 

There are other classical trinomials realizing $S_n$; see \cite[\S 4.4]{Serre-topics}: 

\noindent
-  $P(T,Y)=Y^n-Y^{n-1}-T$ for $p$ not dividing $n(n-1)$, which has branch points $0,\infty, Q(1-(1/n))$ with $Q(Y)=Y^n-Y^{n-1}$, and ramification indices $n$ at $\infty$, $n-1$ at $0$ and $2$ at $Q(1-(1/n))$, 

\noindent
- $P(T,Y)= Y^n-Y-T$ for $p$ not dividing $n(n-1)$; this last example is a special case of (b) below. 

\vskip 2mm

\noindent
(b) ({\it Morse polynomials}): 
$k$ is of characteristic $p\geq 0$ not dividing $n$ and $P(T,Y) = M(Y)-T$ where $M(Y)\in k[Y]$ is a degree $n$ {\it Morse polynomial}, that is: the zeroes $\beta_1,\ldots,\beta_{n-1}$ of the derivative $M^\prime$ are simple and $M(\beta_i)\not=M(\beta_j)$ for $i\not= j$. The branch points of the  cover $f:X\rightarrow \Pp^1$ are $\infty$ and $M(\beta_1),\ldots,M(\beta_{n-1})$, with  ramification indices $n$ at $\infty$ and $2$ at $M(\beta_1),\ldots,M(\beta_{n-1})$. 
See \cite[\S 4.4]{Serre-topics}. 
\vskip 2mm

\noindent
(c) ({\it An example of Uchida}): Let $k$ be any field and $U_0,\ldots, U_3$ be $4$ algebraically independent indeterminates. It is proved in \cite[corollary 2]{uchida} that for every $n\geq 4$, the polynomial $F(Y)= Y^n+U_3Y^3 + U_2Y^2 + U_1Y +U_0$ has Galois group $S_n$ over the field $k(U_0,\ldots, U_3)$. The following lemma makes it possible to derive a polynomial
\vskip 1,5mm 

\centerline{$P(T,Y) = Y^n+u_3(T) Y^3 + u_2(T) Y^2 + u_1(T) Y +u_0(T) \in k[T,Y]$}
\vskip 1,5mm

\noindent
 of Galois group $S_n$ over $\overline k(T)$.

\begin{lemma} \label{lemma:specialization}
Let $\underline U = (U_1,\ldots, U_\ell)$ be a set of algebraically independent indeterminates and $F(\underline U,Y)\in k(\underline U)[Y]$ be a degree $n$ polynomial with Galois group $S_n$ over $\overline k(\underline U)$. Then there exist infinitely many $\ell$-tuples $\underline u_T=(u_1(T),\ldots, u_\ell(T)) \in k[T]^\ell$ such that the polynomial $F(\underline u_T ,Y)$ has Galois group $S_n$ over $\overline k(T)$.
\end{lemma}

\begin{proof}
The polynomial $F(\underline U,Y)$ has Galois group $S_n$ over the field 
$\overline k(T)(\underline U)$. The desired conclusion follows from the Hilbert specialization property of the hilbertian field $\overline k(T)$ but one needs a version that provides good specialisations in $k(T)$ (and not just in $\overline k(T)$). This is classical if $k$ is infinite 
({\it e.g.} \cite[\S 13.2]{FrJa}). For the general case, we resort to theorem 3.3 from 
\cite{De-Hilbert-density} which shows that given a Hilbert subset ${\mathcal H}\subset \overline k(T)$,
for all but finitely many $t_0\in \overline k(T)$, there exists $a\in \overline k(T)$ such that if $b\in k[T]$ is any non-constant polynomial, then ${\mathcal H}$ contains infinitely many elements of the form $t_0+ab^m$ ($m\geq 0$). This gives what we want if $a$ can be chosen in $k(T)$. Although this is not stated, the proof 
shows that such a choice is possible; the main point is to adjust \cite[lemma 3.2]{De-Hilbert-density} 
to show that there are infinitely many cosets of $k(T)$ modulo $\overline k(T)^p$, where $p$ is the characteristic of $k$. \end{proof}

\subsubsection{Special realizations of extensions of PAC fields}
We say a field extension $E/k$ can be {\it realized by a polynomial} $Q(Y)\in k[Y]$ if $Q(Y)$ is the irreducible polynomial over $k$ of some primitive element of $E/k$.






\begin{corollary} \label{cor:trinomial_realization} 
Let $k$ be a PAC field of characteristic $p\geq 0$. If $n\geq 2$ and $p$ does not divide $n(n-1)$, 
every degree $n$ extension $E/k$ can be realized by a trinomial $Y^n - Y + b$ for some $b\in k$. 
Furthermore, if $p\not=2$, the separable closure $k^\sep$ is generated over $k$ by all elements $y\in k^\sep$ such that $y^n-y\in k$ for some integer $n\geq 2$.   
\end{corollary}

\begin{proof}[Proof of corollary \ref{cor:trinomial_realization}] The first part follows from corollary \ref{cor:PAC} applied with $f:X\rightarrow \Pp^1$ given by the trinomial $P(T,Y)= Y^n-Y-T$  from \S \ref{ssec:classical_realizations} (a) and the \'etale algebra $\prod_{l=1}^s E_l/k$ taken to be the  field extension $E/k$. To prove the second part, consider a separable extension $E/k$ of degree $m\geq 2$. Pick an integer $n\geq m$ such that $p$ does not divide $n(n-1)$ (this is possible as $p\not=2$) and do as above but with the \'etale algebra $\prod_{l=1}^s E_l/k$ taken to be the product of the field 
extension $E/k$ with $n-m$ copies of the trivial extension $k/k$. Conclude that $E/k$ has a primitive
element whose irreducible polynomial divides $Y^n-Y+b$ for some $b\in k$. As $E/k$ is an arbitrary finite extension, this provides the announced description of $k^\sep$.
\end{proof}

Proceeding as above but using the Morse polynomial realization (b) from \S \ref{ssec:classical_realizations} (instead of trinomial realizations), we obtain this statement. 

\begin{corollary} \label{cor:morse}
Let $n\geq 2$ be an integer, $k$ be a PAC field of  characteristic $p\geq 0$ not dividing $n$ and $M(Y)\in k[Y]$ be a  degree $n$ Morse polynomial. Then every degree $n$ extension $E/k$ can 
be realized by a polynomial $M(Y)+b$ for some $b\in k$.
\end{corollary}

Finally Uchida's example and lemma \ref{lemma:specialization} from \S \ref{ssec:classical_realizations} (c) yield this.

\begin{corollary} \label{cor:uchida}
Let $n\geq 4$ be an integer and $k$ be a PAC field of any characteristic. Then every separable degree $n$ extension $E/k$ can be realized by a polynomial $Y^n+aY^3+bY^2+cY+d$ for some $a,b,c,d\in k$.
\end{corollary}



\subsubsection{Variants} \label{ssec:trinomials-variants}
\hskip 1mm
\vskip 0,5mm

\noindent
(a) {\it Finite fields}. Proceeding as above but using corollary \ref{cor:finite-fields} instead of 
corollary \ref{cor:PAC} leads to the following conclusions for finite fields:
\vskip 1mm

\noindent
{\it - if $n\geq 2$ and $q\geq (2nn!)^2$ is a prime power with $(q, n(n-1))=1$, the ex\-ten\-sion $\Ff_{q^n}/\Ff_q$ can be realized by a trinomial $Y^n - Y+ b \in \Ff_q[Y]$,} 
\vskip 1mm

\noindent
{\it - if $M(Y)\in \Ff_q[Y]$ is a degree $n$ Morse polynomial such that $(n,q)=1$ and $q \geq (2n n!)^2$, the extension $\Ff_{q^n}/\Ff_q$ can be realized by the polynomial $M(Y)+b$ for some $b\in \Ff_q$.}
\vskip 3mm 

\noindent
(b) {\it $p$-adic fields}. It follows from (a) that 
\vskip 1mm

\noindent
{\it - if $p\geq (2nn!)^2$ is a prime, the degree $n$ unramified extension of $\Qq_p$ can be realized by a trinomial $Y^n - Y+ b$ for some $b\in \Zz_p$, or by a polynomial $M(Y)+b$ with $b\in \Zz_p$ and $M(Y)\in \Zz_p[Y]$ a degree $n$ monic polynomial with reduction modulo $p$ a Morse polynomial in $\Ff_p[Y]$}. 
\vskip 1mm

\noindent
This can also be proved by using \S \ref{ssec:local}
instead of \S \ref{ssec:finite_fields} (with possibly another bound on $p$).

\vskip 3mm

\noindent
(c) {\it Other trinomials}. The trinomials $Y^n - Y^{n-1}- T$ and $Y^n-T^r Y^m + T^s$ from \S \ref{ssec:classical_realizations} 
can be used instead of $Y^n - Y- T$ to provide similar conclusions. The assumption on $p$ remains that $p\not| \hskip 0pt n(n-1)$ for the former, and for the latter, it is that  $p \not|  mn(n-m)$ (with the other conditions on $n$ and $m$ from \S \ref{ssec:classical_realizations}); and the bound on $q$ can be replaced by the better one $q=p^f\geq (6n!)^2$. 

\vskip 3mm

\noindent
(d) {\it Missing characteristics}. Given an integer $n\geq 2$ and a prime $p$, corollary \ref{cor:PAC}, combined with lemma \ref{lemma:specialization}, shows in fact that
\vskip 1mm

\noindent
(*) {\it every degree $n$ separable extension $E/k$ of a PAC field $k$ of characteristic $p$ can be realized by some trinomial $Y^n + aY^m + b$ for some integer $1\leq m <n$ and some $a,b\in k$}, 
\vskip 1mm

\noindent
provided that the following holds:
\vskip 1mm

\noindent
(**) {\it there exists $1\leq m <n$ such that the trinomial $Y^n + UY^m + V$ has Galois group $S_n$ over $\overline \Ff_p(U,V)$ (where $U, V$ are two indeterminates).}
\vskip 1mm

\noindent
There are many results about condition (**) in the literature, notably in the papers \cite{uchida}, \cite{cohen2} and \cite{cohen3}. Here are conclusions that can be derived about the cases not covered 
by corollary \ref{cor:trinomial_realization}:
\vskip 0,5mm

\noindent
- if $p\not=2$, ($p|n$ or $p|n-1$) and $n$ is odd, (**) holds with $Y^n+UY^2 + V$ or with $Y^n-UY+V$
(\cite[corollary 3]{cohen3} and \cite[theorem2]{uchida}),
\vskip 0,5mm

\noindent
- if $p=2$ and $n$ is odd, (**) holds with $Y^n+UY^2 + V$ if $n\geq 5$
\cite[corollary 3]{cohen3} and with $Y^n-UY + V$ if $n=3$ \cite[theorem2]{uchida},
\vskip 0,5mm

\noindent
- if $p=3$ and $n=4$, (**) holds with $Y^n-UY + V$ \cite[theorem2]{uchida},
\vskip 0,6mm

\noindent
- if ($p=5$ and $n=6$) or ($p=2$ and $n=6$), (**) does not hold: $Y^6-UY + V$ has Galois group ${\rm PGL}_2(\Ff_5)$ over $\Ff_5(U,V)$ and $Y^6-UY + V$ has Galois group $A_5$ over $\Ff_4(U,V)$ \cite{uchida}.
\vskip 1mm

\noindent
Note that conjoining these results with corollary \ref{cor:trinomial_realization}, we obtain that  (**) and (*) always hold if $n$ is odd.
\vskip 3mm

\noindent
(e) {\it Number fields}. Over a number field $k$, extensions with trinomial realizations are more sparse. For example, Angeli proved that, for every $n\geq 3$, there are (up to some standard equivalence for trinomials) only finitely many degree $n$ trinomials with coefficients in $k$, irreducible and with Galois group a primitive subgroup $G\subset S_n$ distinct from $S_n$ and $A_n$ \cite{Angeli_these}. See also \cite{Angeli} where the same is proved with ``$G\subset S_n$ primitive'' replaced by ``$G$ solvable'' in the case $n$ is a prime. 


\subsection{Hurwitz spaces} \label{ssec:hurwitz_spaces}
Given an integer $r\geq 3$ and a finite group ${G}$ (resp. a subgroup $G\subset S_n$), 
there is a coarse moduli space called {\it Hurwitz space} for G-covers of $\Pp^1$ of group ${G}$ (resp. for mere covers of $\Pp^1$ of degree $n$ and geometric monodromy group $G\subset S_n$) with $r$ branch points. 
We view it here as a (reducible) variety defined over $\Qq$;
it can be more generally defined as a scheme over some extension ring of $\Zz[1/|G|]$. We do not distinguish between the G-cover and mere cover situations and use the same notation $\hbox{\sf{H}}_{r}({G})$ for the Hurwitz space. 

A central moduli property is that for any field $k$ of characteristic $0$, there is a one-one correspondence between the set of $\overline{k}$-rational points on $\hbox{\sf{H}}_{r}({G})$ and the set of iso\-mor\-phisms classes of (G- or mere) covers defined over $\overline{k}$ with the given invariants. Furthermore for every closed point $[f]\in \hbox{\sf{H}}_{r}({G})$, the field $k([f])$ is the field of moduli of the corresponding (G- or mere) cover $f$. We refer to \cite{DeDo1} for more on fields of moduli; in standard situations ({\it e.g.} $Z(G)=\{1\}$ for G-covers, ${\rm Cen}_{S_n}(G)=\{1\}$ for  mere covers) and in most situations below, the field of moduli is a field of definition of $f$ and is the smallest one.

Denote by $\hbox{\sf{U}}_r$ the configuration space for finite subsets of $\Pp^1$ of cardinality $r$. The map $\Psi_{r}: \hbox{\sf{H}}_{r}({G}) \rightarrow
\hbox{\sf{U}}_r$ that sends each isomorphism class of cover $[f]$
 in $\hbox{\sf{H}}_{r}({G})$ to its branch point set ${\bf
t}\in \hbox{\sf{U}}_r$ is an \'etale cover defined over $\Qq$. The geometrically irreducible components of $\hbox{\sf{H}}_{r}({G})$ correspond to the connected components of $\hbox{\sf{H}}_{r}({G})\otimes_\Qq\Cc$, which in turn correspond to the orbits of the so-called {\it Hurwitz monodromy action}, of the fundamental group of $\hbox{\sf{U}}_r$ (the {\it Hurwitz group $\mathcal{H}_r$}) on a fiber $\Psi_{r}^{-1}({\bf t})$ (${\bf t}\in \hbox{\sf{U}}_r(\overline k)$). For more on Hurwitz spaces, see \cite{Vo96}
or \cite{De_zakopane}.

In this situation we have this result. In (b) (ii) where $k$ is a number field and $v$ is a place of $k$, we use the notation $k_v^{\ur,f}$ ($f\in \Nn, f>0$) for the unramified extension of $k_v$ of degree $f$.

\begin{corollary} \label{cor:hurwitz_spaces}
Let $\hbox{\sf{H}}$ be a component of $\hbox{\sf{H}}_{r}({G})$ defined over a field 
$k$ and such that the restriction $(\Psi_{r})_{\hbox{$\scriptstyle \sf{H}$}}: \hbox{\sf{H}} \rightarrow \hbox{\sf{U}}_r$ induces a $k$-mere cover of geometric monodromy group $S_N$ with $N=\deg((\Psi_{r})_{\hbox{$\scriptstyle \sf{H}$}})$. 
\vskip 1,5mm

\noindent
{\rm (a)} If $k$ is PAC of characteristic $0$ and $\prod_{l=1}^s E_l/k$ a $k$-\'etale algebra with $\sum_{l=1}^s [E_l:k] = N$, there exists a Zariski-dense subset ${\mathcal U}\subset \hbox{\sf{U}}_r(k)$ such that for each ${\bf t}_0 \in {\mathcal U}$, the $k$-\'etale algebra $\prod_{l=1}^s E_l/k$ is the collection of the smallest fields of definition of the $\overline k$-covers $[f:X\rightarrow \Pp^1]$ in $\hbox{\sf{H}}$  (mere or G-) with branch divisor ${\bf t}_0$.

\vskip 2mm

\noindent
{\rm (b)} If $k$ is a number field, there exist two constants $p(r,{G})$ and $q(r,{G})$ depending only on $r$ and ${G}$ with the following property. Let $S$ be a finite subset of finite places of $k$ with residue field of order $\geq q(r,{G})$ and residue characteristic $\geq p(r,{G})$, and for each 
$v\in S$, let $d_{v,1}\ldots, d_{v,s_v}$ be positive integers with $\sum_{l=1}^{s_v} d_{v,l}= N$. 
There exists a Zariski-dense subset ${\mathcal U} \subset \hbox{\sf{U}}_r(k)$, of the form ${\mathcal U}= \hbox{\sf{U}}_r(k) \cap \prod_{v\in S} U_v$ for some $v$-adic open subsets $U_v\subset \hbox{\sf{U}}_r(k_v)$, such that for each ${\bf t}_0 \in {\mathcal U}$, the $\overline k$-covers $f:X\rightarrow \Pp^1$ in $\hbox{\sf{H}}$ 
with branch divisor ${\bf t}_0$ satisfy the following:
\vskip 0,5mm



\noindent
\hskip 2mm {\rm (i)} their field of moduli $k([f])$ is a degree $N$ extension of $k$, and 
\vskip 0,5mm

\noindent
\hskip 2mm {\rm (ii)} for each $v\in S$, the $k_v$-\'etale algebra $\prod_{l=1}^{s_v} k_v^{\ur,d_{v,l}}/k_v$  is the collection of the smallest fields of definition of the $\overline{k_v}$-covers $f\otimes_{\overline k} \overline{k_v}$ (for any given embedding ${\overline k} \hookrightarrow \overline{k_v}$).
\end{corollary}

\begin{proof}
(a) and (b) respectively follow from the twisting lemma \ref{prop:twisted cover}, applied to the 
$k$-mere cover $(\Psi_{r})_{\hbox{$\scriptstyle \sf{H}$}}: \hbox{\sf{H}} \rightarrow \hbox{\sf{U}}_r$, and from
the interpretation recalled above of the specializations of this cover at some point ${\bf t}_0\in \hbox{\sf{U}}_r$ as the fields of moduli of the points $[f] \in \hbox{\sf{H}}$ above ${\bf t}_0$. Over a PAC field, the field of moduli is always a field of definition (and is the smallest one) \cite{DeDo1}. The PAC situation in (a) offers no further difficulty. 

Statement (b) is a local-global statement as in \S \ref{ssec:field_loc-glob}. Conditions $p_v\geq p(r,{G})$ and $q_v\geq q(r,{G})$ ($v\in S$) are here to guarantee
assumptions {\rm (good-red)} and {\rm ($\kappa$-big-enough)} of corollary \ref{cor:loc-glob}. Condition
{\rm (good-red)} also implies that the field of moduli of each local cover $f\otimes_{\overline k} \overline{k_v}$ is a field of definition \cite{DeHa}. The variety $\hbox{\sf{U}}_r$ being birational to $\Pp^r$ has the weak approximation property {\rm (weak-approx\hskip2pt$/S$)} for any set $S$, so sets of the form ${\mathcal U}= \hbox{\sf{U}}_r(k) \cap \prod_{v\in S} U_v$ with $U_v\subset \hbox{\sf{U}}_r(k_v)$ non-empty $v$-adic open subsets,  are non-empty, and even Zariski-dense in  $\hbox{\sf{U}}_r(k)$.
The standard trick recalled in \S \ref{ssec:standard_trick} should be used to obtain condition (ii)
that the extension $k([f])/k$ be exactly of degree $N$.
\end{proof}

There is in corollary \ref{cor:hurwitz_spaces} the assumption that $(\Psi_{r})_{\hbox{$\scriptstyle \sf{H}$}}: \hbox{\sf{H}} \rightarrow \hbox{\sf{U}}_r$ be a $k$-mere cover of geometric monodromy group $S_N$. 
This assumption can be checked in practical situations. Indeed the geometric monodromy group is the image group of the Hurwitz monodromy action (restricted to the component $\hbox{\sf{H}}$), which 
can be made totally explicit. 


\bibliography{FCAmore1}
\bibliographystyle{alpha}

\end{document}